\documentclass[a4paper,12pt,reqno]{amsart}

\usepackage{amsmath}
\usepackage{amssymb}
\usepackage{amsfonts}
\usepackage{graphicx}
\usepackage[colorlinks]{hyperref}
\renewcommand\eqref[1]{(\ref{#1})} 

\graphicspath{ {images/} }
\title[Blow-up solutions of damped Klein-Gordon equation]{Blow-up solutions of damped Klein-Gordon equation on the Heisenberg group}

\author[Michael Ruzhansky]{Michael Ruzhansky}
\address{\href{www.ruzhansky.org}{Michael Ruzhansky:}
	\endgraf
	Department of Mathematics: Analysis, Logic and Discrete Mathematics
	\endgraf
	Ghent University, Belgium
	\endgraf
	and
	\endgraf
	School of Mathematical Sciences
	\endgraf Queen Mary University of London 
	\endgraf
	United Kingdom
	\endgraf
	{\it E-mail address} {\rm Michael.Ruzhansky@ugent.be}
}

\author[Bolys Sabitbek]{Bolys Sabitbek}
\address{ \href{http://analysis-pde.org/bolys-sabitbek/}{Bolys Sabitbek:}
	\endgraf
	School of Mathematical Sciences
	\endgraf Queen Mary University of London 
	\endgraf
	United Kingdom
	\endgraf 
	and
	\endgraf 
	Al-Farabi Kazakh National University 
	\endgraf 
	Almaty,
	Kazakhstan 
	\endgraf
	{\it E-mail address} {\rm b.sabitbek@qmul.ac.uk}
}


\thanks{The first and second authors were supported by EPSRC grant EP/R003025/2. The first author was also supported by FWO Odysseus 1 grant G.0H94.18N: Analysis and Partial Differential Equations and the Methusalem programme of the Ghent University Special Research Fund (BOF) (Grant number 01M01021).}

\newtheoremstyle{theorem}
{10pt}          
{10pt}  
{\sl}  
{\parindent}     
{\bf}  
{. }    
{ }    
{}     
\theoremstyle{theorem}

\numberwithin{equation}{section}
\theoremstyle{plain}
\newtheorem{thm}{Theorem}[section]

\theoremstyle{definition}
\newtheorem{defn}[thm]{Definition}
\newtheorem{rem}[thm]{Remark}

\newtheoremstyle{defi}
{10pt}          
{10pt}  
{\rm}  
{\parindent}     
{\bf}  
{. }    
{ }    
{}     
\theoremstyle{defi}

\begin{document}
	\begin{abstract}
		In this note, we prove the blow-up of solutions of the semilinear damped Klein-Gordon equation in a finite time for arbitrary positive initial energy on the Heisenberg group. This work complements the paper \cite{RT-18} by the first author and Tokmagambetov, where the global in time well-posedness was proved for the small energy solutions.
	\end{abstract}
	\maketitle
	\section{Introduction}
\subsection{Setting of the problem} This note is devoted to study the blow up of solutions of the Cauchy problem for the semilinear damped Klein-Gordon equation for the sub-Laplacian $\mathcal{L}$ on the Heisenberg group $\mathbb{H}^n$:
\begin{align}\label{Wave-problem}
\begin{cases}
u_{tt}(t) - \mathcal{L} u(t) + bu_t(t)+ mu(t)= f(u), & t>0, \\ 
u(x,0)   = u_0(x),\,\,\, & u_0 \in H^1_{\mathcal{L}}(\mathbb{H}^n),\\
u_t(x,0) = u_1(x),\,\,\, & u_1 \in L^2(\mathbb{H}^n),
\end{cases}
\end{align}
with the damping term determined by $b>0$ and the mass $m>0$. A total energy of problem \eqref{Wave-problem} is defined as
\begin{align*}
E(t)&=\frac{1}{2}|| u_t||^2_{L^2(\mathbb{H}^n)} + \frac{m}{2}||u||^2_{L^2(\mathbb{H}^n)}+ \frac{1}{2}||\nabla_{H} u||^2_{L^2(\mathbb{H}^n)} - \int_{\mathbb{H}^n}F(u) dx,
\end{align*}
where we assume the following 
\begin{align}\label{Cond-F}
&	g: [0,\infty] \rightarrow \mathbb{R},\nonumber\\
&	F(z) = g(|z|) \,\, \text{  for } \,\,  z \in \mathbb{C}^n, \\
&   f(z) = \frac{g'(|z|)z}{|z|}. \nonumber
\end{align}
Then we have that
\begin{align*}
	\frac{\partial}{\partial\varepsilon} F(z+ \varepsilon \xi) |_{\varepsilon=0} &= \frac{\partial}{\partial \varepsilon} g(|z + \varepsilon \xi|)|_{\varepsilon =0} \\
	& = g'(|z+\varepsilon \xi| )\frac{\partial}{\partial \varepsilon}(|z+\varepsilon \xi|)|_{\varepsilon=0} \\
	&= \frac{g'(|z|)}{|z|} \frac{1}{2}(\overline{z}\xi + z \overline{\xi}) \\
	&={\rm Re } \left( f(z) \overline{\xi}\right),
\end{align*}
and
\begin{align*}
	\frac{\partial}{\partial x_j}F(u(x)) &= g'(|u(x)|) \frac{\partial|u(x)| }{\partial x_j}\\
	&= \frac{g'(|u(x)|)}{2|u(x)|} \left(u(x)\frac{\partial \overline{u}(x) }{\partial x_j}  +\frac{\partial u(x) }{\partial x_j}\overline{u}(x)\right)\\
	&= {\rm Re }\left( f(u(x)) \frac{\partial \overline{u}(x)}{\partial x_j}\right).
\end{align*}
The conservation of energy law follows from
\begin{align*}
\frac{\partial E(t)}{\partial t} & = \frac{\partial }{\partial t}\left[ \frac{1}{2} || u_t||^2_{L^2(\mathbb{H}^n)} +\frac{m}{2}||u||^2_{L^2(\mathbb{H}^n)}+ \frac{1}{2}||\nabla_{H} u||^2_{L^2(\mathbb{H}^n)}  - \int_{\mathbb{H}^n} F(u) dx \right]\\
& = {\rm Re }\int_{\mathbb{H}^n } \overline{u}_t [u_{tt} + m u - \mathcal{L}u - f(u)] dx\\
& = - b \int_{\mathbb{H}^n } |u_t|^2 dx,
\end{align*} 
this gives
\begin{equation}\label{eq-energy}
E(t) + b\int_{0}^t|| u_s(s) ||^2_{L^2(\mathbb{H}^n)}ds =E(0),
\end{equation}
where 
\begin{equation*}
E(0):=\frac{1}{2} || u_1||^2_{L^2(\mathbb{H}^n)} +\frac{m}{2}||u_0||^2_{L^2(\mathbb{H}^n)}+ \frac{1}{2}||\nabla_{H} u_0||^2_{L^2(\mathbb{H}^n)}  - \int_{\mathbb{H}^n} F(u_0) dx.
\end{equation*}
Also, let us define the Nehari functional 
\begin{align*}
I(u) = m||u||^2_{L^2(\mathbb{H}^n)} + ||\nabla_{H} u||^2_{L^2(\mathbb{H}^n)}  - {\rm Re } \int_{\mathbb{H}^n}f(u)\overline{u} dx.
\end{align*}
We assume that the nonlinear term $f(u)$ satisfies the condition
\begin{equation*}
	f(0)=0, \,\,\, \text{  and } \,\, \alpha F(u) \leq {\rm Re } [ f(u)\overline{u}],
\end{equation*} 
where $\alpha >2$. In particular, this includes the case
\begin{equation*}
	f(u) = |u|^{p-1}u\,\,\,\text{  for } p>1.
\end{equation*}
\subsection{Literature overview } 
The study of the damped wave equation on the Heisenberg group started in Bahouri-Gerrard-Xu \cite{BGX-00} to prove the dispersive and Strichartz inequalities based on the analysis in Besov-type spaces. Later, Greiner-Holcman-Kannai \cite{GHK-02} explicitly computed the wave kernel for the class of second-order subelliptic operators, where their class contains degenerate elliptic and hypoelliptic operators such as the sub-Laplacian and the Grushin operator. Also, M\"uller-Stein \cite{MS-99} established $L^p$-estimates for the wave equation on the Heisenberg group. Recently, M\"uller-Seeger \cite{MS-15} obtained the sharp version of $L^p$ estimates on the $H$-type groups. 

The blow-up solutions of evolution equations on the Heisenberg group were considered by Georgiev-Palmieri \cite{GP-20} where they proved the global existence and nonexistence results of the Cauchy problem for the semilinear damped wave equation on the Heisenberg group with the power nonlinear term. The proof of blow-up solutions is based on the test function method. The first author and Yessirkegenov \cite{RY-22} established the existence and non-existence of global solutions for semilinear heat equations and inequalities on sub-Riemannian manifolds. In \cite{RY-22-1}, by using the comparison principle they obtain blow-up type results and global in $t$-boundedness of solutions of nonlinear equations for the heat $p$-sub-Laplacian on the stratified Lie groups. The global existence and nonexistence for the nonlinear porous medium equation were studied by the authors in \cite{RST-21} on the stratified Lie groups.

This work is motivated by the paper \cite{RT-18} of the first author and Tokmagambetov where the global existence of solutions for small data of problem \eqref{Wave-problem} was shown on the Heisenberg group and on general graded Lie groups. In the sense of the potential wells theory, we can understand this result in the sense that when the initial energy is less than the mountain pass level $E(0)<d$ and the Nehari functional is positive $I(u_0)>0$, there exists a global solution of the problem \eqref{Wave-problem}. A natural question arises when the solution of problem \eqref{Wave-problem} blows up in a finite time or $E(0)>0$ and $I(u_0)<0$. 

The main aim of this paper is to obtain the blow-up solutions of problem \eqref{Wave-problem} in a finite time for arbitrary positive initial energy. Our proof is based on an adopted concavity method, which was introduced by Levine \cite{Levine-74} to establish the blow-up solutions of the abstract wave equation of the form $Pu_{tt}=-Au + F(u)$ (including the Klein-Gordon equation) for the negative initial energy. It was also used for parabolic type equations (see \cite{Levine-74-1, Levine73, Levine-90, LP-74, LP2-74}). Modifying the concavity method, Wang \cite{W-08} proved the nonexistence of global solutions to nonlinear damped Klein-Gordon equation for arbitrary positive initial energy under sufficient conditions. Later, Yang-Xu \cite{YX-18} extended this result by introducing a new auxiliary function and the adopted concavity method.

\subsection{Preliminaries on the Heisenberg group} Let us give a brief introduction of the Heisenberg group.
Let $\mathbb{H}^n$ be the Heisenberg group, that is, the set $\mathbb{R}^{2n+1}$ equipped with the group law 
\begin{equation*}
\xi \circ \widetilde{\xi} := (x + \widetilde{x}, y + \widetilde{y}, t + \widetilde{t}+2 \sum_{i=1}^{n}(\widetilde{x}_i y_i - x_i \widetilde{y}_i )),
\end{equation*}
where $\xi:= (x,y,t) \in \mathbb{H}^n$, $x:=(x_1,\ldots,x_n)$, $y:=(y_1,\ldots,y_n)$, and $\xi^{-1}=-\xi$ is the inverse element of $\xi$ with respect to the group law. The dilation operation of the Heisenberg group with respect to the group law has the following form (see e.g. \cite{FR-book}, \cite{RS_book})
\begin{equation*}
\delta_{\lambda}(\xi) := (\lambda x, \lambda y, \lambda^2 t) \,\, \text{for}\,\, \lambda>0.
\end{equation*} 
The Lie algebra $\mathfrak{h}$ of the left-invariant vector fields on the Heisenberg group $\mathbb{H}^n$ is spanned by 
\begin{equation*}
X_i:= \frac{\partial }{\partial x_i} + 2y_i\frac{\partial }{\partial t} \,\, \text{for} \,\, 1\leq i \leq n,
\end{equation*}
\begin{equation*}
Y_i:= \frac{\partial }{\partial y_i} - 2x_i\frac{\partial }{\partial t} \,\, \text{for} \,\, 1\leq i \leq n,
\end{equation*}
and  with their (non-zero) commutator
\begin{equation*}
[X_i,Y_i]= - 4 \frac{\partial}{\partial t}.
\end{equation*} 
The horizontal gradient of $\mathbb{H}^n$ is given by 
\begin{equation*}
\nabla_{H}:= (X_1,\ldots,X_n,Y_1,\ldots,Y_n),
\end{equation*} 
so the sub-Laplacian on $\mathbb{H}^n$ is given by
\begin{equation*}
\mathcal{L}:= \sum_{i=1}^{n} \left(X_i^2 + Y_i^2\right). 
\end{equation*} 

\begin{defn}[Weak solution]
	A function 
	\begin{align*}
	&u \in C ([0,T_1);H^1_{\mathcal{L}}(\mathbb{H}^n)) \cap C^1 ([0,T_1);L^2(\mathbb{H}^n)), \\
	&u_t \in L^2 ([0,T_1);H^1_{\mathcal{L}}(\mathbb{H}^n)),\\
	&u_{tt} \in L^2 ([0,T_1);H^{-1}_{\mathcal{L}}(\mathbb{H}^n)), 
	\end{align*}
	satisfying
	\begin{equation}\label{eq-weak}
	{\rm Re} \langle u_{tt}, v \rangle + {\rm Re}\int_{\mathbb{H}^n } \nabla_{H} u \cdot \nabla_{H} v dx +m{\rm Re} \int_{\mathbb{H}^n } uv dx + b  {\rm Re}\int_{\mathbb{H}^n } u_tv dx\ =  {\rm Re}\int_{\mathbb{H}^n } f(u)v dx,
	\end{equation}
	for all $v \in H^1_{\mathcal{L}}(\mathbb{H}^n)$ and a.e. $t \in [0,T_1)$ with $u(0)=u_0(x)$ and $u_t(0)=u_1(x) $ represents a weak solution of problem \eqref{Wave-problem}.
\end{defn}

Note that $T_1$ denotes the lifespan of the solution $u(x,t)$ and $\langle \cdot, \cdot \rangle$ is the duality between $H^{-1}_{\mathcal{L}}(\mathbb{H}^n)$ and $H^{1}_{\mathcal{L}}(\mathbb{H}^n)$. Here $H^1_{\mathcal{L}}(\mathbb{H}^n)$ denotes the sub-Laplacian Sobolev space, analysed by Folland \cite{Fol-75}, see also \cite{FR-book}.

\section{Main Result}

We now present the main result of this paper.

\begin{thm}\label{thm_main}
		Let $b>0$, $m>0$ and $\mu= \max\{ b,m, \alpha \}$. Assume that nonlinearity $f(u)$ satisfies 
	\begin{equation}\label{eq-2.1}
	\alpha F(u) \leq {\rm Re}[ f(u)\overline{u}]\,\,\, \text{  for }\,\, \alpha >2,
	\end{equation}
	where $F(u)$ is as in \eqref{Cond-F}.
	Assume that the Cauchy data $u_0 \in H_{\mathcal{L}}^1(\mathbb{H}^n)$ and $u_1 \in L^2(\mathbb{H}^n)$  satisfy
	\begin{equation}\label{eq-2.2}
	I(u_0) =  m||u_0||^2_{L^2(\mathbb{H}^n)} + ||\nabla_{H} u_0||^2_{L^2(\mathbb{H}^n)} - {\rm Re}\int_{\mathbb{H}^n}\overline{u}_0f(u_0) dx<0,
	\end{equation}
	and 
	\begin{equation}\label{eq-thm-cond}
	 {\rm Re}(u_0,u_1)_{L^2(\mathbb{H}^n)}\geq \frac{\alpha(\mu+1)}{m(\alpha-2)}E(0).
	\end{equation}
	Then the solution of equation \eqref{Wave-problem} blows up in finite time $T^*$ such that 
	\begin{equation*}
		0< T^* \leq \frac{2(\mu+1)(bT_0 +1)}{(\alpha-2)(\mu+1-m)} \frac{||u_0||^2_{L^2(\mathbb{H}^n)}}{{\rm Re}(u_0,u_1)},
	\end{equation*}
	where  the blow-up time $T^* \in (0,T_0)$ with $ T_0<+\infty$.
\end{thm}
\begin{rem} \begin{itemize}
		\item[(i)] Note that we have times $T^*$, $T_0$ and $T_1$. The relationship between this times is the blow-up time $T^*\in (0,T_0) \subset (0,T_1)$ where $T_0<+\infty$ and $T_1 = +\infty$.
		\item[(ii)]	The local existence for the Klein-Gordon equation was shown in \cite{Caz-85} and \cite{CH-98}. The global in time well-posedness of problem \eqref{Wave-problem} was proved by the first author and Tokmagambetov \cite{RT-18} for the small energy solutions and the nonlinearity $f(u)$ satisfying 
		\begin{align*}
		|f(u) - f(v)| \leq C (|u|^{p-1}+|v|^{p-1})|u-v|,
		\end{align*}
		with $1<p\leq 1+1/n$.
	\end{itemize}

\end{rem}
\begin{proof}[Proof of Theorem \ref{thm_main} ]
	First, recall the Nehari functional 
	\begin{align*}
	I(u) = m||u||^2_{L^2(\mathbb{H}^n)} + ||\nabla_{H} u||^2_{L^2(\mathbb{H}^n)}  -  {\rm Re}\int_{\mathbb{H}^n}\overline{u}f(u) dx.
	\end{align*}
Then the proof includes two steps.	
	
	\textbf{Step I.}
	In this step, we claim that 
	\begin{equation*}
	I(u(t))<0, \,\,\, \text{  and }	\,\,\, A(t) > \frac{2\alpha(\mu+1)}{m(\alpha-2)}E(0),
	\end{equation*}
	for $0\leq t <T_1$ where $\mu= \max\{ b,m, \alpha \}$ and
	\begin{align*}
	A(t) = 2 {\rm Re}(u,u_t) + b ||u||^2_{L^2(\mathbb{H}^n)}.
	\end{align*}
	By using \eqref{eq-weak} along with $v =\overline{u}$ we get
	\begin{align}\label{eq-*}
	A'(t) &= 2 || u_t||^2_{L^2(\mathbb{H}^n)} + 2{\rm Re }\langle u_{tt},u \rangle + 2b {\rm Re}\int_{\mathbb{H}^n } \overline{u}u_t dx \nonumber\\
	&= 2 || u_t||^2_{L^2(\mathbb{H}^n)} - 2 I(u), \,\,\, 0\leq t <T_1.
	\end{align}
	In the last line we have used that
	\begin{align*}
	{\rm Re}	\langle u_{tt}, u \rangle  &=  {\rm Re}\int_{\mathbb{H}^n } f(u)\overline{u} dx- \int_{\mathbb{H}^n } |\nabla_{H} u|^2 dx - m\int_{\mathbb{H}^n } |u|^2 dx - b  {\rm Re}\int_{\mathbb{H}^n } \overline{u} u_t dx\\
		&= - I(u) - b {\rm Re}\int_{\mathbb{H}^n } \overline{u}u_t dx.
	\end{align*}
	Now let us suppose by contradiction that 
	\begin{equation*}
	I(u(t))<0 \,\,\, \text{  for all } \,\, 0\leq t<t_0,
	\end{equation*}
	and 
	\begin{equation*}
	I(u(t_0))=0.
	\end{equation*}
	Hereafter $0<t_0<T_1$. It is easy to see that $A'(t)>0$ over $[0,t_0)$ and 
	\begin{align}\label{eq-A}
	A(t) > A(0) \geq 2 {\rm Re}(u_0,u_1) \geq  \frac{2\alpha(\mu+1)}{m(\alpha-2)}E(0).
	\end{align}
	Since $u(t)$ and $u_t(t)$ are both continuous in $t$ that gives 
	\begin{align}\label{eq-2.6}
	A(t_0) \geq  \frac{2\alpha(\mu+1)}{m(\alpha-2)}E(0).
	\end{align}
	Next we need to show a contradiction to \eqref{eq-2.6}. Using \eqref{eq-energy} and \eqref{eq-2.1}, we have
	\begin{align*}
	E(0) &= E(t) + b \int_{0}^t||u_s||^2_{L^2(\mathbb{H}^n)}ds \\
	& =\frac{1}{2} || u_t||^2_{L^2(\mathbb{H}^n)}+ \frac{m}{2}||u||^2_{L^2(\mathbb{H}^n)} + \frac{1}{2} ||\nabla_{H} u||^2_{L^2(\mathbb{H}^n)} \\
	&- \int_{\mathbb{H}^n } F(u)dx + b \int_{0}^t||u_s||^2_{L^2(\mathbb{H}^n)}ds\\
	&\geq \frac{1}{2} || u_t||^2_{L^2(\mathbb{H}^n)} + \frac{m}{2}||u||^2_{L^2(\mathbb{H}^n)} + \frac{1}{2} ||\nabla_{H} u||^2_{L^2(\mathbb{H}^n)} \\
	&- \frac{1}{\alpha} {\rm Re}\int_{\mathbb{H}^n } \overline{u} f(u)dx + b \int_{0}^t||u_s||^2_{L^2(\mathbb{H}^n)}ds\\
	& = \frac{1}{2} || u_t||^2_{L^2(\mathbb{H}^n)}  + \frac{1}{\alpha}I(u) + \left( \frac{m}{2} -\frac{m}{\alpha} \right)||u||^2_{L^2(\mathbb{H}^n)}\\
	&+\left(\frac{\alpha -2}{2\alpha} \right)||\nabla_{H} u||^2_{L^2(\mathbb{H}^n)} + b\int_{0}^t||u_s||^2_{L^2(\mathbb{H}^n)}ds.
	\end{align*}
	If we use $I(u(t_0))=0$ and $\frac{m(\alpha -2)}{\alpha(\mu+1)}<1$, then
	\begin{align}\label{eq-EA}
	E(0) &  \geq \frac{1}{2} || u_t(t_0)||^2_{L^2(\mathbb{H}^n)} + \frac{m(\alpha -2)}{2\alpha}||u(t_0)||^2_{L^2(\mathbb{H}^n)}\nonumber\\
	& \geq  \frac{m(\alpha -2)}{2\alpha(\mu +1)} \left( || u_t(t_0)||^2_{L^2(\mathbb{H}^n)}  + (\mu +1) ||u(t_0)||^2_{L^2(\mathbb{H}^n)} \right)\nonumber\\
	& > \frac{m(\alpha -2)}{2\alpha(\mu +1)} \left( 2 {\rm Re}(u(t_0),u_t(t_0)) + \mu ||u(t_0)||^2_{L^2(\mathbb{H}^n)} \right)\nonumber \\
	& \geq  \frac{m(\alpha -2)}{2\alpha(\mu +1)} A(t_0).
	\end{align}
	Note that for the strict inequality above we use that the assumption \eqref{eq-2.2} implies that $||u_0||_{L^2(\mathbb{H}^n)} \neq 0$.
	 We have also used the fact $a^2 + b^2 -2ab \geq 0$, where $a = || u_t(t_0)||_{L^2(\mathbb{H}^n)}$ and $b= || u(t_0)||_{L^2(\mathbb{H}^n)}$. It gives the contradiction to \eqref{eq-2.6}. This proves our claim. 
	
	\textbf{Step II.} Define the functional
	\begin{align*}
	M(t) = ||u||^2_{L^2(\mathbb{H}^n)} + b \int_{0}^t ||u(s)||^2_{L^2(\mathbb{H}^n)} ds + b (T_0-t)||u_0||^2_{L^2(\mathbb{H}^n)}, 
	\end{align*}
	for $0\leq t \leq T_0$. Then
	\begin{align*}
	M'(t) &= 2 {\rm Re}(u,u_t) +b||u(t)||^2_{L^2(\mathbb{H}^n)}-b||u_0||^2_{L^2(\mathbb{H}^n)}  \\
	&=2 {\rm Re}(u,u_t) +2b  \int_{0}^t  {\rm Re}(u(s),u_s(s))ds,
	\end{align*}
	since 
	\begin{align*}
		\int_{0}^t \frac{d}{ds} || u (s)||_{L^2(\mathbb{H}^n)}^2 ds = || u(t)||_{L^2(\mathbb{H}^n)}^2 -  || u(0)||_{L^2(\mathbb{H}^n)}^2.
	\end{align*}
	We observe the following estimates
	\begin{align*}
	|{\rm Re}(u,u_t)|^2 &\leq ||u_t||^2_{L^2(\mathbb{H}^n)} || u||^2_{L^2(\mathbb{H}^n)}, \\
	\left( \int_{0}^t |{\rm Re}(u(s),u_s(s))|ds  \right)^2& \leq \left(\int_{0}^t || u(s)||^2_{L^2(\mathbb{H}^n)}ds \right)\left(\int_{0}^t || u_s(s)||^2_{L^2(\mathbb{H}^n)}ds\right),
	\end{align*}
	and
	\begin{align*}
	2 {\rm Re}(u,u_t)\int_{0}^t  {\rm Re}(u(s),u_s(s))ds &\leq 2 ||u||_{L^2(\mathbb{H}^n)} ||u_t||_{L^2(\mathbb{H}^n)}\\
	&\times  \left(\int_{0}^t || u(s)||^2_{L^2(\mathbb{H}^n)}ds\right)^{1/2}\left( \int_{0}^t || u_s(s)||^2_{L^2(\mathbb{H}^n)}ds\right)^{1/2}\\
	\leq ||u||^2_{L^2(\mathbb{H}^n)} &\int_{0}^t || u_s(s)||^2_{L^2(\mathbb{H}^n)}ds + ||u_t||^2_{L^2(\mathbb{H}^n)} \int_{0}^t || u(s)||^2_{L^2(\mathbb{H}^n)}ds.
	\end{align*}
	Using the above inequalities, we calculate 
	\begin{align*}
	(M'(t))^2 &= 4 \left( |{\rm Re}(u,u_t)|^2 + 2b {\rm Re}(u,u_t) \int_0^t  {\rm Re}(u(s),u_s(s))ds + b^2 \left(\int_0^t  {\rm Re}(u(s),u_s(s))ds\right)^2 \right) \\
	&\leq 4 \left( ||u||^2_{L^2(\mathbb{H}^n)} + b \int_{0}^t ||u(s)||^2_{L^2(\mathbb{H}^n)} ds  \right)\left( ||u_t||^2_{L^2(\mathbb{H}^n)} + \int_{0}^t ||u_s(s)||^2_{L^2(\mathbb{H}^n)} ds \right),
	\end{align*}
	for all $0\leq t \leq T_0$. The second derivate with respect to time of $M(t)$ is
	\begin{align*}
	M''(t) = 2||u_t||^2_{L^2(\mathbb{H}^n)} - 2 I(u),
	\end{align*}
	for all $0\leq t\leq T_0$, where we used the equality from \eqref{eq-*}.  Then we construct the differential inequality as follows
	\begin{align*}
	&M''(t) M(t) - \frac{\omega+3}{4}  (M'(t))^2 \geq M(t)\left( M''(t) - (\omega +3)\left( ||u_t||^2 + b\int_{0}^t ||u_s(s)||^2_{L^2(\mathbb{H}^n)} ds \right) \right)\\
	&=M(t) \left( -(\omega+1) ||u_t||^2_{L^2(\mathbb{H}^n)} - (\omega +3)b\int_{0}^t ||u_s(s)||^2_{L^2(\mathbb{H}^n)} ds -2I(u)\right),
	\end{align*}
	where we assume that $\omega >1$. We shall now show that the following term is nonnegative 
	\begin{align*}
	\eta(t) &= -(\omega+1) ||u_t||^2_{L^2(\mathbb{H}^n)} - (\omega +3)b\int_{0}^t ||u_s(s)||^2_{L^2(\mathbb{H}^n)} ds -2I(u)\\
	&\geq (\alpha -\omega -1) ||u_t||^2_{L^2(\mathbb{H}^n)} + b(2\alpha - \omega -3) \int_{0}^t ||u_s(s)||^2_{L^2(\mathbb{H}^n)} ds\\
	& + m(\alpha -2)||u||^2_{L^2(\mathbb{H}^n)} + (\alpha -2)||\nabla_{H} u||^2_{L^2(\mathbb{H}^n)} -2\alpha E(0) \\
	& =  (\alpha -\omega -1) \left[||u_t||^2_{L^2(\mathbb{H}^n)}  + (b+1) ||u||^2_{L^2(\mathbb{H}^n)}\right] + (\alpha -2)||\nabla_{H} u||^2_{L^2(\mathbb{H}^n)} -2\alpha E(0) \\
	&+ b(2\alpha - \omega -3) \int_{0}^t ||u_s(s)||^2_{L^2(\mathbb{H}^n)} ds + (m(\alpha -2 )-  (b+1) (\alpha - \omega -1)|)||u||^2_{L^2(\mathbb{H}^n)} \\
	& \geq  (\alpha -\omega -1) \left[2{\rm Re}(u,u_t)  + b ||u||^2_{L^2(\mathbb{H}^n)}\right] + (\alpha -2)||\nabla_{H} u||^2 _{L^2(\mathbb{H}^n)}-2\alpha E(0) \\
	&+ b(2\alpha - \omega -3) \int_{0}^t ||u_s(s)||^2_{L^2(\mathbb{H}^n)} ds + (m(\alpha -2) -  (b+1) (\alpha - \omega -1)|)||u||^2_{L^2(\mathbb{H}^n)}.
	\end{align*}
	In the second line that we have used \eqref{eq-EA}.
	By selecting $\omega = \alpha-1 - \frac{m(\alpha -2)}{\mu+1}$ which satisfies $\omega>1$ since $\mu +1>m$ and using the argument from Step I, we obtain
	\begin{align*}
	\eta(t)& > \frac{m(\alpha -2)}{\mu+1}  (2 {\rm Re}(u,u_t)  + b ||u||^2_{L^2(\mathbb{H}^n)}) -2\alpha E(0)\\
	&>\frac{m(\alpha -2)}{\mu+1}   (2 {\rm Re}(u_0,u_1)  + b ||u_0||^2_{L^2(\mathbb{H}^n)}) -2\alpha E(0) \\
	& > \left(\frac{m(\alpha -2)}{\mu+1}\right)   2{\rm Re}(u_0,u_1) -2\alpha E(0)\\
	& \geq 0,
	\end{align*}
 Note that we have used the fact $A'(t)>0$ and the expression \eqref{eq-A} with $A(t)=2 {\rm Re}(u,u_t)  + b ||u||^2_{L^2(\mathbb{H}^n)}$, and the condition \eqref{eq-thm-cond} in the last line, respectively. So we obtain the inequality
	\begin{equation*}
	M''(t) M(t) - \frac{\omega+3}{4}  (M'(t))^2 > 0.
	\end{equation*}
	Then 
	\begin{equation*}
	\frac{d}{dt} \left[ \frac{M'(t)}{M^{\frac{\omega+3}{4}}(t)} \right] > 0  \Rightarrow 	\begin{cases}
	M'(t) \geq \left[ \frac{M'(0)}{M^{\frac{\omega+3}{4}}(0)} \right] M^{\frac{\omega+3}{4}}(t),\\
	M(0)=(bT_0+1)|| u_0||_{L^2(\mathbb{H}^n)}^2.
	\end{cases}
	\end{equation*}
Let us denote $\sigma = \frac{\omega -1}{4}$. Then we have 
\begin{equation*}
	-\frac{1}{\sigma} \left[ M^{-\sigma}(t) - M^{-\sigma}(0) \right] \geq \frac{M'(0)}{M^{\sigma +1}(0)} t,
\end{equation*}
that gives
\begin{align*}
	M(t) \geq \left( \frac{1}{M^{\sigma}(0)} - \frac{\sigma M'(0)}{M^{\sigma+1}(0)}t \right)^{-\frac{1}{\sigma}}.
\end{align*}
Then the blow-up time $T^*$ satisfies
\begin{equation*}
	0<T^* \leq \frac{M(0)}{\sigma M'(0)},
\end{equation*}
where $M'(0)= 2{\rm Re} (u_0,u_1)$.	This completes the proof.
\end{proof}
			

\begin{thebibliography}{NZW01}
			
			
			\bibitem{BGX-00}
			Bahouri H., Gerard P., Xu C,J.: Spaces de Besov et estimations de Strichartz généralisées sur le groupe de Heisenberg. \textit{J. Anal. Math.}, 82, 93--118 (2000) 
			
		\bibitem{Caz-85}
		Cazenave T.: Uniform estimates for solutions of nonlinear Klein-Gordon equations. \textit{Journal of Functional Analysis}, 60, 36--55 (1985)
			\bibitem{CH-98}
			Cazenave T., Haraux A.: An introduction to semilinear evolution equations. Oxford Lecture Series in Mathematics and its Applications, 13. The Clarendon Press, Oxford University Press, New York, 1998
			
			
			
			
			
	
			
		
			
			\bibitem{GHK-02}
			Greiner P.C., Holcman D., Kannai Y.: Wave kernels related to second-order operators. \textit{Duke Math. J.}, 114 (2), 329--386 (2002)
			
			
			

			
			\bibitem{GP-20}
			Georgiev V., Palmieri A.: Critical exponent of Fujita-type for the semilinear damped wave equation on the Heisenberg group with power nonlinearity.
		\textit{J. Differential Equations}, 269, no. 1, 420--448 (2020)
		\bibitem{Fol-75}
	 Folland 	G. B.: Subelliptic estimates and function spaces on nilpotent Lie groups.\textit{ Ark. Mat.}, 13(2), 161--207 (1975)
	 
	 \bibitem{FR-book}
	 Fischer V., Ruzhansky M.: Quantization on nilpotent Lie groups, volume 314 of Progress in Mathematics. Birkh\"auser/Springer, [Open access book], 2016
	 \bibitem{FS-82}
	 Folland G. B., Stein E. M.: Hardy spaces on homogeneous groups, volume 28 of Mathematical Notes. Princeton University Press, Princeton, N.J.; University of Tokyo Press, Tokyo, 1982
			\bibitem{Levine-74}
			Levine H. A.: Some additional remarks on nonexistence of global solutions to nonlinear wave equations of the form $Pu_{tt}=-Au + \mathcal{F}(u)$. \textit{Trans. Amer. Math. Soc.}, 55, 52--72 (1974)
			\bibitem{Levine-74-1}
			Levine H. A.: A note on a nonexistence theorem for some nonlinear wave equations. \textit{SIAM J. Math. Anal.}, 5, 138--146 (1974)
			\bibitem{Levine-90}
			Levine H. A.: The role of critical exponents in blow-up theorems. \textit{SIAM Rev.}, 32,  262--288 (1990)
			
			\bibitem{Levine73}
			Levine H. A.: 
			\newblock Some nonexistence and instability theorems for formally parabolic equations of the form $Pu_t=-Au +\mathcal{F}(u)$. \textit{Arch. Ration. Mech. Anal.}, 51, 277--284 (1973)
			
			\bibitem{LP-74}
			Levine H. A., Payne L. E.: Nonexistence theorems for the heat equation with nonlinear boundary conditions and for the porous medium equation backward in time. \textit{J. Differential Equations}, 16, 319--334 (1974)
			
			\bibitem{LP2-74}
			Levine H. A., Payne L. E.: Some nonexistence theorems for initial-boundary value problems with nonlinear boundary constraints. \textit{Proc. Amer. Math. Soc.}, 46, 277--284 (1974)
		
			
		
		
		
		\bibitem{MS-99} M\"uller D., Stein E.M.: $L^p$-estimates for the wave equation on the Heisenberg group. \textit{Rev. Mat. Iberoam.}, 15 (2), 297--334 (1999)
		\bibitem{MS-15} M\"uller D., Seeger A.: Sharp $L^p$ bounds for the wave equation on groups of Heisenberg type. \textit{Anal. PDE}, 8 (5)
		1051--1100 (2015)
		
		
		
			
			\bibitem{PY-20}
			Pang Y., Yang Y.: A note on finite time blow-up for dissipative Klein-Gordon equation. \textit{Nonlinear Analysis}, 195, 111729 (2020)
			\bibitem{PS-75}
			Payne L. E., Sattinger D. H.: Saddle points and instability of nonlinear hyperbolic equations. \textit{Israel J. Math.}, 22, 273--303 (1975)
			
		
			
			\bibitem{RST-21}
			Ruzhansky M., Sabitbek B., Torebek B.: Global existence and blow-up of solutions to porous medium equation and pseudo-parabolic equation, I. Stratified Groups. \textit{Manuscripta Math.,} to appear, (2021)
			\bibitem{RS_book}
			Ruzhansky M., Suragan D.: Hardy inequalities on homogeneous groups. \textit{Progress in Math.} Vol. 327, Birkh\"auser, 588 pp, 2019. (open access book)
			\bibitem{RT-18}
			Ruzhansky M., Tokmagambetov N.: Nonlinear damped wave equations for the sub-Laplacian on the Heisenberg group and for Rockland operators on graded Lie groups. \textit{J. Differential Equations}, 265, 5212--5236 (2018) 
			\bibitem{RY-22}
			Ruzhansky M., Yessirkegenov N.: Existence and non-existence of global solutions for semilinear heat equations and inequalities on sub-Riemannian manifolds, and Fujita exponent on unimodular Lie groups. \textit{J. Differential Equations}, 308, 455--473 (2022)
			
			 	\bibitem{RY-22-1}
			 Ruzhansky M., Yessirkegenov N.: A comparison principle for higher order nonlinear hypoelliptic heat operators on graded Lie groups. \textit{Nonlinear Analysis}, 215, 112621 (2022)
			
			\bibitem{Sattinger-68}
			Sattinger D. H. On global solution of nonlinear hyperbolic equations. \textit{Arch. Rat. Mech. Anal.}, 30, 148--172 (1968)
			
			
			
		
			
		
			
			\bibitem{XD-13}
			Xu R., Ding Y.: Global solutions and finite time blow up for damped Klein-Gordon equation. 
			\textit{Acta Mathematica Scientia}, 33B(3), 643--652 (2013)
			
			\bibitem{XZh-17}
			Xu R. Z., Zhang M. Y., Chen  S.H., Yang Y.B., Shen J. H.: The initial-boundary value problems for a class of six order nonlinear wave equation. \textit{Discrete Contin. Dyn. Syst.}, 37 (11), 5631--5649 (2017)
			
		
			
			
			\bibitem{YX-18}
			Yang Y., Xu R.: Finite time blowup for nonlinear Klein–Gordon equations with arbitrarily positive initial energy. \textit{Applied Mathematics Letters,} 77, 21--26 (2018)
			\bibitem{Zh-02}
			Zhang J.: 
			Sharp conditions of global existence of nonlinear Schr\"odinger and Klein-Gordon equations. 
			\textit{Nonlinear Analysis}, 48, 191--207 (2002)
			
			\bibitem{W-08}
			Wang Y.: A sufficient condition for finite time blow-up of the nonlinear Klein-Gordon equations with arbitrary positive initial energy. 
			\textit{Proc. Amer. Math. Soc.}, 136, 3477--3482 (2008)
		\end{thebibliography}
\end{document}